\documentclass[leqno,english]{smfart}
\usepackage{hyperref}
\usepackage{amsthm}
\usepackage{amssymb}
\usepackage[square,numbers]{natbib}
\usepackage[all]{xy}

\renewcommand{\thesection}{\arabic{section}}
\renewcommand{\thesubsubsection}{\thesection.\arabic{subsubsection}}
\numberwithin{subsubsection}{section}

\renewcommand{\theenumi}{\alph{enumi}}
\renewcommand{\labelenumi}{(\theenumi)}

\numberwithin{equation}{section}
\renewcommand{\theequation}{\arabic{equation}}

\theoremstyle{plain}

\newtheorem{mainthm}{Theorem}

\newtheorem*{maincor}{Corollary}

\newtheorem{lemm}[subsubsection]{Lemma}
\newtheorem{prop}[subsubsection]{Proposition}

\newtheorem{fact}[subsubsection]{Fact}

\DeclareMathOperator{\A}{\mathcal{A}}
\DeclareMathOperator{\C}{\mathcal{C}}

\DeclareMathOperator{\M}{\mathcal{M}}
\DeclareMathOperator{\Op}{\mathcal{O}}

\DeclareMathOperator{\haut}{\mathrm{haut}}
\DeclareMathOperator{\Map}{\mathrm{Map}}
\DeclareMathOperator{\Mor}{\mathrm{Mor}}
\DeclareMathOperator{\Hom}{\mathrm{Hom}}

\DeclareMathOperator{\Ho}{\mathrm{Ho}}
\DeclareMathOperator{\sk}{\mathit{sk}}

\DeclareMathOperator*{\colim}{\mathrm{colim}}

\DeclareMathOperator{\NN}{\mathbb{N}}
\DeclareMathOperator{\ZZ}{\mathbb{Z}}
\DeclareMathOperator{\kk}{\Bbbk}

\DeclareMathOperator{\QQ}{\mathbb{Q}}

\DeclareMathOperator{\AOp}{\mathsf{A}}
\DeclareMathOperator{\BOp}{\mathsf{B}}
\DeclareMathOperator{\COp}{\mathsf{C}}
\DeclareMathOperator{\DOp}{\mathsf{D}}
\DeclareMathOperator{\EOp}{\mathsf{E}}
\DeclareMathOperator{\FOp}{\mathsf{F}}
\DeclareMathOperator{\GOp}{\mathsf{G}}
\DeclareMathOperator{\HOp}{\mathsf{H}}
\DeclareMathOperator{\IOp}{\mathsf{I}}

\DeclareMathOperator{\LOp}{\mathsf{L}}
\DeclareMathOperator{\POp}{\mathsf{P}}
\DeclareMathOperator{\QOp}{\mathsf{Q}}

\title[On mapping spaces of differential graded operads]{On mapping spaces\\of differential graded operads\\with the commutative operad as target}

\author{Benoit Fresse}
\date{15 September 2009}
\address{UMR 8524 du CNRS et de l'Universit\'e de Lille 1 - Sciences et Technologies\\
Cit\'e Scientifique -- B\^atiment M2\\
F-59655 Villeneuve d'Ascq C\'edex (France)}
\email{Benoit.Fresse@math.univ-lille1.fr}
\urladdr{http://math.univ-lille1.fr/\~{ }fresse}
\thanks{Research supported in part by grant ANR-06-JCJC-0042 ``OBTH''}

\begin{document}

\begin{abstract}
The category of differential graded operads is a cofibrantly generated model category
and as such inherits simplicial mapping spaces.
The vertices of an operad mapping space are just operad morphisms.
The $1$-simplices represent homotopies between morphisms in the category of operads.

The goal of this paper is to determine the homotopy of the operadic mapping spaces $\Map_{\Op_0}(\EOp_n,\COp)$
with a cofibrant $E_n$-operad $\EOp_n$ on the source
and the commutative operad $\COp$ on the target.
First,
we prove that the homotopy class of a morphism $\phi: \EOp_n\rightarrow\COp$
is uniquely determined by a multiplicative constant
which gives the action of $\phi$ on generating operations
in homology.
From this result,
we deduce that the connected components of $\Map_{\Op_0}(\EOp_n,\COp)$
are in bijection with the ground ring.
Then
we prove that each of these connected components is contractible.

In the case $n = \infty$,
we deduce from our results that the space of homotopy self-equivalences of an $E_\infty$-operad
in differential graded modules has contractible connected components
indexed by the invertible elements of the ground ring.
\end{abstract}

\maketitle

\section*{Introduction}

Recall that any model category $\A$ inherits simplicial mapping spaces $\Map_{\A}(A,X)$
such that $\pi_0(\Map_{\A}(A,X))$ is identified with the morphism set $[A,X]_{\Ho\A}$
of the homotopy category of~$\A$ (see~\cite{DwyerKanFunction}).

The purpose of this paper is to study mapping spaces of operads
in dg-modules,
for any fixed ground ring $\kk$ (for short, we use the prefix dg to mean differential graded).
To be precise,
we deal with the category of non-unitary operads,
the operads $\POp$ such that $\POp(0) = 0$.
This category, denoted by $\Op_0$,
inherits a full model structure from the base category of dg-modules
(see~\cite{BergerMoerdijk,HinichHomotopy}).

\medskip
In the context of simplicial sets
and topological spaces,
the operad of commutative monoids is defined termwise by the terminal object of the category.
Accordingly,
any mapping space with this operad as target is automatically
contractible.
This is no more the case of the commutative operad
in dg-modules $\COp$
since we have then $\COp(r) = \kk\not=0$, for each arity $r>0$.
Nevertheless
the next results, which give the main objectives of the paper, show that the commutative operad still satisfies strong rigidity properties
in the dg-setting:

\begin{mainthm}\label{Result:MappingSpaces}
Let $\POp_n$ be a cofibrant $E_n$-operad ($n = 1,2,\dots,\infty$).
We have
\begin{equation*}
\pi_0(\Map_{\Op_0}(\POp_n,\COp)) = \kk
\quad\text{and}
\quad\pi_i(\Map_{\Op_0}(\POp_n,\COp),\phi) = *\quad\text{when $i>0$},
\end{equation*}
for every choice of base point $\phi\in\Map_{\Op_0}(\POp_n,\COp)_0$.
\end{mainthm}

In the context of dg-modules,
an $E_n$-operad refers to an operad weakly-equivalent to the chain operad of Boardman-Vogt little $n$-cubes for $n<\infty$,
to the commutative operad for $n = \infty$.
In Theorem~\ref{Result:MappingSpaces},
we use tacitely a non-unitary version of the notion of an $E_n$-operad
for which the term of arity zero
is set to be $0$.
This convention, contrary to the usual definition,
is used throughout the article.

The commutative operad $\COp$
is generated as an operad by an operation $\mu\in\COp(2)$
which represents the structure product of commutative algebras.
The convention~$\COp(0) = 0$
implies the existence of operad morphisms $\rho_{c}: \COp\rightarrow\COp$,
naturally associated to all $c\in\kk$,
such that $\rho_{c}(\mu) = c\cdot\mu$.
The identity $\pi_0(\Map_{\Op_0}(\POp_n,\COp)) = \kk$
of Theorem~\ref{Result:MappingSpaces}
comes from the possibility of composing a base point $\phi\in\Map_{\Op_0}(\POp_n,\COp)_0$
with these rescaling morphisms $\rho_{c}: \COp\rightarrow\COp$, for $c\in\kk$.

\medskip
The space of self-maps $\Map_{\A}(A,A)$ of a cofibrant-fibrant object $A$
in a cofibrantly generated model category $\A$
forms a simplicial monoid.
The simplicial set $\haut_{\A}(A)$
formed by the connected components of $\Map_{\A}(A,A)$
which are invertible in $\pi_0(\Map_{\A}(A,A))$
defines the space of homotopy automorphisms of~$A$ (see~\cite{DwyerKanLocalization}).
The connected components of~$\haut_{\A}(A)$
are all weakly-equivalent (by vertex multiplication).

In the case $n=\infty$,
Theorem~\ref{Result:MappingSpaces}
gives as an easy corollary:

\begin{maincor}
Let $\POp_\infty$ be a cofibrant $E_\infty$-operad.
We have:
\begin{equation*}
\pi_0(\haut_{\Op_0}(\POp_{\infty})) = \kk^{\times}\quad\text{and}\quad\pi_i(\haut_{\Op_0}(\POp_{\infty})) = *\quad\text{in degree $i>0$}.
\end{equation*}
\end{maincor}

To obtain this corollary,
we use simply that the augmentation of a cofibrant $E_\infty$-operad $\POp_{\infty}\xrightarrow{\sim}\COp$
induces a weak-equivalence of mapping spaces
\begin{equation*}
\Map_{\Op_0}(\POp_{\infty},\POp_{\infty})\xrightarrow{\sim}\Map_{\Op_0}(\POp_{\infty},\COp)
\end{equation*}
and, under the identity $\pi_0(\Map_{\Op_0}(\POp_{\infty},\POp_{\infty})) = \pi_0(\Map_{\Op_0}(\POp_{\infty},\COp)) = \kk$,
the multiplication of connected components in $\Map_{\Op_0}(\POp_{\infty},\POp_{\infty})$
corresponds to scalar multiplications in $\kk$ (use the detailed analysis of the concluding section of the paper).

The homotopy automorphism groups $\pi_*(\haut_{\Op_0}(\POp_n))$
seem more intricate for $n<\infty$.
Some hints come from the case $n=2$:
the classifying spaces of pure braids define the underlying collection of an $E_2$-operad;
the Grotendieck-Teichm\"uller group,
whose elements realize universal automorphisms of braided monoidal categories,
acts on these classifying spaces by operad morphisms
and one conjectures (see~\cite{KontsevichMotives})
that this action defines an embedding of the Grotendieck-Teichm\"uller group
into $\pi_0(\haut_{\Op_0}(\POp_2))$
(see~\cite{Tamarkin} for a result in this direction in the characteristic zero setting).

\medskip
The result of Theorem~\ref{Result:MappingSpaces} and its corollary also hold in the simplicial setting
because the normalization functor from simplicial modules to dg-modules
induces the right-adjoint of a Quillen equivalences between simplicial operads and dg-operads
(adapt the line of argument of~\cite[Proposition I.4.4 and Theorem II.5.4]{Quillen}).
The definition of mapping spaces is easier in the simplicial context,
but the crux of the proof of Theorem~\ref{Result:MappingSpaces}
relies on constructions of the dg-context.

In a sense,
this paper represents a first application of results of~\cite{FresseEnKoszulDuality}
because the proof of Theorem~\ref{Result:MappingSpaces}
is based on a certain cofibrant model of~$E_n$-operads
defined in that article.

This alluded to cofibrant model has the form of an operadic cobar construction $\POp_n = \BOp^c(\DOp_n)$,
where $\DOp_n = \Lambda^{-n}\EOp_n^{\vee}$ is the operadic desuspension
of the dual cooperad of an $E_n$-operad $\EOp_n$
satisfying mild requirements.
We use a natural filtration of the operadic cobar construction $\BOp^c(\Lambda^{-n}\EOp_n^{\vee})$
to produce a decomposition of the mapping space $\Map_{\Op_0}(\BOp^c(\Lambda^{-n}\EOp_n^{\vee}),\COp)$
into a tower of fibrations with Eilenberg-Mac Lane spaces as fibers.
To obtain the result of Theorem~\ref{Result:MappingSpaces},
we just observe that the extended homotopy spectral sequence of this tower fibrations practically vanishes
at $E_1$-stage.

The alluded to filtration of the operadic cobar construction $\BOp^c(\Lambda^{-n}\EOp_n^{\vee})$
is deduced from a filtration of quasi-free operads by arity of generators.
The definition of this filtration is reviewed in~\S\ref{HomotopySpectralSequence}.
The decomposition of mapping spaces arising from such a filtration
is defined in the same section~(\S\ref{HomotopySpectralSequence})
and Theorem~\ref{Result:MappingSpaces} is established afterwards (in~\S\ref{HomotopyGroups}).

In the concluding section,
we study applications of Theorem~\ref{Result:MappingSpaces}
to the definition of operad mophisms $\phi: \BOp^c(\Lambda^{-n}\EOp_n^{\vee})\rightarrow\COp$
and $\phi^{\sharp}: \Lambda^{n-1}\LOp_{\infty}\rightarrow\EOp_n$,
where $\LOp_{\infty}$ is a model of an $L_{\infty}$-operad (an operad equivalent to the operad of Lie algebras).
In brief,
we prove that such morphisms are characterized, within the homotopy category of operads,
by their effect in homology.
In characteristic zero,
the existence of morphisms of the form $\phi^{\sharp}: \Lambda^{n-1}\LOp_{\infty}\rightarrow\EOp_n$
has been used for associating a deformation complex to $E_n$-algebra structures
arising from solutions of the Deligne conjecture (see~\cite{KontsevichMotives} for a comprehensive account of these ideas).

Before beginning,
we review some main conventions used throughout the article.

\section*{Conventions and background}

In the sequel,
we adopt conventions and notation of the papers~\cite{FresseCylinder,FresseEnKoszulDuality}
which give the operadic background of this work.
In this section,
we just review some overall conventions on dg-modules and operads.

\medskip
Throughout this paper, a dg-module refers to a lower $\ZZ$-graded module $C$, over a fixed ground ring $\kk$,
together with a differential $\delta: C\rightarrow C$
that decreases degrees by $1$.
The category of dg-modules, denoted by $\C$,
is equipped with its usual tensor product $\otimes: \C\times\C\rightarrow\C$
together with the symmetry isomorphism $\tau: C\otimes D\rightarrow D\otimes C$
involving a sign.
The notation $\pm$ is used to represent any sign which arises from an application of this symmetry isomorphism.

The morphism sets of any category $\A$ are denoted by $\Mor_{\A}(A,X)$.
The internal hom-objects of the category of dg-modules are denoted by $\Hom_{\C}(C,D)$.
Recall that a homogeneous element of $\Hom_{\C}(C,D)$
is just a morphism of $\kk$-modules $f: C\rightarrow D$
that increases degrees by $d = \deg(f)$.
The differential of $f$ in $\Hom_{\C}(C,D)$
is defined by the graded commutator of $f$ with the internal differentials of $C$ and $D$.
The elements of the dg-hom $\Hom_{\C}(C,D)$
are called homomorphisms to be distinguished from the actual morphisms of dg-modules $f\in\Mor_{\C}(C,D)$.

The category of dg-modules $\C$
is equipped with its standard model structure
for which the weak-equivalences are the morphisms which induce an isomorphism in homology,
the fibrations are the degreewise surjections (see~\cite[\S 2.3]{Hovey}).

\medskip
As explained in the introduction, we use the notation $\Op_0$ to refer to the category of non-unitary operads,
the operads $\POp$ such that $\POp(0) = 0$.
The unit operad,
which defines the initial object of~$\Op_0$,
is denoted by~$\IOp$.
The category $\Op_0$ inherits a model structure such that a morphism $\phi: \POp\rightarrow\QOp$
is a weak-equivalence (respectively, fibration)
if its components $\phi: \POp(r)\rightarrow\QOp(r)$, $r\in\NN$,
are weak-equivalences (respectively, fibrations)
in the category of dg-modules
(detailed recollections and comprehensive bibliographical references on this background can be found in~\cite[\S 1.3]{FresseCylinder}).
The cofibrations are characterized by the right-lifting-property with respect to acyclic fibrations.

\section[The extended homotopy spectral sequence of operadic mapping spaces]{Quasi-free operads\\
and the extended homotopy spectral sequence\\
of operadic mapping spaces}\label{HomotopySpectralSequence}

The simplicial mapping spaces $\Map_{\A}(A,X)$ in a cofibrantly generated model category $\A$
are defined by morphism sets $\Mor_{\A}(A\otimes\Delta^n,X)$
where $A\otimes\Delta^{\bullet}$
is a cosimplicial objet associated to $A$, a cosimplicial frame of $A$,
so that:
\begin{enumerate}
\item\label{FrameVertices}
we have an identity $A\otimes\Delta^{0} = A$;
\item\label{FrameCofibration}
the morphisms $\eta_i: A\otimes\Delta^{0}\rightarrow A\otimes\Delta^{n}$
induced by the embeddings $\eta_i: \{i\}\rightarrow\{0<\dots<n\}$
in the simplicial category $\Delta$
assemble to a Reedy cofibration $\ell^{\bullet} A\rightarrowtail A\otimes\Delta^{\bullet}$,
where~$\ell^{\bullet} A$ is a cosimplicial object such that
$\ell^n A = \amalg_{i=0}^{n} A$;
\item\label{FrameHomotopy}
the morphism $\epsilon: A\otimes\Delta^{n}\rightarrow A\otimes\Delta^{0}$
induced by the constant map $\eta_i: \{0<\dots<n\}\rightarrow\{0\}$
is a weak-equivalence in $\A$.
\end{enumerate}
(We refer to~\cite[\S 1, \S 5]{Hovey} for full details on this definition and its applications.)
Requirements (\ref{FrameVertices}-\ref{FrameHomotopy})
ensure that the simplicial set $\Map_{\A}(A,X)$
satisfies reasonable homotopy invariance properties
when we restrict ourself to cofibrant objects on the source and fibrant objects on the target.
The first requirement (\ref{FrameVertices})
gives an identity between the vertices of the mapping space $\phi\in\Map_{\A}(A,X)_0$
and the morphisms of the category $\phi\in\Mor_{\A}(A,X)$.
The $1$-simplices $\psi\in\Map_{\A}(A,X)_1$
can also be identified with left-homotopies between morphisms in $\A$
because the assumptions imply that $A\otimes\Delta^{1}$
forms a cylinder-object associated to $A$.

In the category of dg-modules $\A = \C$,
we have a natural cosimplicial framing, associated to each cofibrant object $C\in\C$,
defined by the tensor products $C\otimes N_*(\Delta^n)$,
where $N_*(\Delta^n)$ is the normalized chain complex of the $n$-simplex $\Delta^n$.
In this setting,
the mapping space $\Map_{\C}(C\otimes N_*(\Delta^{\bullet}),D)$
forms naturally a simplicial $\kk$-module and the normalized chain complex of this simplicial $\kk$-module
can formally be identified
with the dg-hom of the category of dg-modules $\Hom_{\C}(C,D)$.

\medskip
There is a dual definition of simplicial mapping spaces $\Map_{\A}(A,X)$
in terms of morphism sets $\Mor_{\A}(A,X^{\Delta^n})$
associated to simplicial frames $X^{\Delta^n}$
satisfying the dual of the requirements (\ref{FrameVertices}-\ref{FrameHomotopy})
of cosimplicial frames.
These dual definitions produce weakly-equivalent simplicial mapping spaces
provided that we restrict ourself to cofibrant objects on the source
and fibrant objects on the target.

In the context of operads $\A = \Op_0$,
we apply this dual definition because the functoriality of mapping spaces on the source
is easier to handle when we take a simplicial frame on the target rather than a cosimplicial frame on the source -- indeed, the morphism
$f^*: \Map_{\A}(B,X)\rightarrow\Map_{\A}(A,X)$ induced by $f: A\rightarrow B$
is just given by the composition with $f$
in the morphism sets $\Mor_{\A}(-,X^{\Delta^n})$.
The cofibrant objects that we consider are structures, called quasi-free operads,
defined by the addition of a twisting derivation $\partial: \FOp(M)\rightarrow\FOp(M)$
to the natural differential of a free operad $\FOp(M)$
so that we have a new operad in dg-modules $\POp = (\FOp(M),\partial)$
with the same underlying graded object
as the free operad $\FOp(M)$.
We observe that a cofibrant operad $\POp = (\FOp(M),\partial)$
inherits a natural filtration by arity of generators (under mild assumptions on the twisting homomorphism).
The goal of this section is to study the decomposition of operadic mapping spaces
arising from such natural filtrations on the source.

First of all,
we review the definition of a quasi-free operad
in detail.

\subsubsection{The definition of twisted operads}\label{HomotopySpectralSequence:TwistedOperads}
We borrow the formalism of~\cite[\S 1.4]{FresseCylinder}
for the definition of twisted objects
in the category of operads.

Recall briefly that a collection of homomorphisms $\partial: \POp(n)\rightarrow\POp(n)$
defines an operad derivation $\partial: \POp\rightarrow\POp$
if each $\partial$ commutes with the action of permutations on $\POp(n)$
and we have the derivation relation
\begin{equation}\label{DerivationRelation}
\partial(p\circ_i q) = \partial(p)\circ_i q + \pm p\circ_i\partial(q)
\end{equation}
with respect to the operad composition structure $\circ_i: \POp(m)\otimes\POp(n)\rightarrow\POp(m+n-1)$,
where the sign $\pm$ arises from the standard conventions of differential graded algebra.
The derivation relation implies that $\partial$
cancels the operad unit $1\in\POp(1)$.

A twisting derivation $\partial: \POp\rightarrow\POp$ is an operad derivation of degree $-1$
whose components $\partial: \POp(n)\rightarrow\POp(n)$
satisfy the equation
\begin{equation}\label{TwistingEquation}
\delta(\partial) + \partial^2 = 0
\end{equation}
in $\Hom_{\C}(\POp(n),\POp(n))$, for all $n\in\NN$.

Equation~(\ref{TwistingEquation}) implies that the addition of $\partial: \POp(n)\rightarrow\POp(n)$
to the internal differential $\delta: \POp(n)\rightarrow\POp(n)$
defines a new differential on $\POp(n)$
since we have identities $(\delta+\partial)^2 = \delta^2 + \delta\partial + \partial\delta + \partial^2 = 0 + \delta(\partial) + \partial^2 = 0$.
Hence,
we have a new dg-module associated to each $\POp(n)$ with the same underlying graded module as $\POp(n)$
but the homomorphism $\delta+\partial: \POp(n)\rightarrow\POp(n)$
as differential.
Usually,
we just use the notation of the pair~$(\POp(n),\partial)$ to refer to this twisted dg-module associated to $\POp(n)$.

The derivation relation~(\ref{DerivationRelation})
implies that the composition products of the operad $\POp$
define morphisms of dg-modules between the twisted objects $(\POp(n),\partial)$.
Hence,
the collection of twisted dg-modules $(\POp(n),\partial)$ inherits an operad composition structure when $\partial$ is an operad twisting derivation
so that we have a new operad in dg-modules $(\POp,\partial)$
with the same underlying graded object as $\POp$.

\subsubsection{Recollections on quasi-free operads}\label{HomotopySpectralSequence:QuasiFreeOperads}
A quasi-free operad is a twisted operad $\QOp = (\FOp(M),\partial)$
associated to a free operad $\POp = \FOp(M)$.

The free operad $\FOp(M)$
is defined by the left-adjoint of the obvious forgetful functor $U: \Op\rightarrow\M$
from the category of operads $\Op$
to the category $\M$
formed by collections $M(n)$, $n\in\NN$,
where $M(n)$
is a dg-module equipped with an action of the symmetric group in $n$-letters $\Sigma_n$.
In the sequel,
we use the terminology of $\Sigma_*$-object to refer the objects of this category $\M$.
Recall that the category of $\Sigma_*$-objects
inherits dg-modules of homomorphisms $\Hom_{\M}(M,N)$:
a homomorphism $f\in\Hom_{\M}(M,N)$
is simply a collection of homomorphisms of dg-modules $f\in\Hom_{\C}(M(n),N(n))$
commuting with the action of symmetric groups;
the differential of $\Hom_{\M}(M,N)$ is defined componentwise by the differential of dg-module homomorphisms.

Intuitively,
the free operad $\FOp(M)$ is defined by the collection of dg-modules $\FOp(M)(n)$
spanned by formal operadic composites of generating elements $\xi_i\in M(n_i)$.
In this representation,
we identify the generating $\Sigma_*$-object $M$
with a subobject of the free operad $\FOp(M)$.

In the case of a free operad $\POp = \FOp(M)$,
the derivation relation~(\ref{DerivationRelation}) of~\S\ref{HomotopySpectralSequence:TwistedOperads}
implies that any operad derivation $\partial: \FOp(M)\rightarrow\FOp(M)$
is uniquely determined by a homomorphism $\theta: M\rightarrow M$
so that $\theta = \partial|_{M}$.
In the sequel,
we adopt the notation $\partial = \partial_{\theta}$
for the derivation associated to $\theta: M\rightarrow M$.
In~\cite[Proposition 1.4.5]{FresseCylinder},
we observe that equation~(\ref{TwistingEquation}) of~\S\ref{HomotopySpectralSequence:TwistedOperads}
holds if and only if we have the equation
\begin{equation}\label{FreeTwistingEquation}
\delta(\theta) + \partial_{\theta}\cdot\theta = 0
\end{equation}
in $\Hom_{\M}(M,\FOp(M))$.

The adjunction relation $\FOp: \M\rightleftarrows\Op :U$
asserts that a morphism $\phi: \FOp(M)\rightarrow\QOp$
towards an operad $\QOp$
is uniquely determined by a morphism of $\Sigma_*$-objects $f: \M\rightarrow\QOp$
so that $f = \phi|_{M}$.
In our intuitive definition of the free operad,
we simply use the commutation relation $\phi(p\circ_i q) = \phi(p)\circ_i\phi(q)$
to determine the map $\phi$ on the formal operadic composites of~$\FOp(M)$
from its restriction $f = \phi|_{M}$.
In the case of a quasi-free operad $\POp = (\FOp(M),\partial_{\theta})$,
the obtained morphism $\phi = \phi_f$
does not necessarily preserve the differential of the quasi-free object $\POp = (\FOp(M),\partial_{\theta})$.
Therefore
we extend the construction of $\phi = \phi_f$
to homomorphisms $f\in\Hom_{\M}(M,\QOp)$.
In this setting,
we have a homomorphism $\phi_f: \FOp(M)\rightarrow\QOp$,
preserving grading, symmetric group action and composition structure,
naturally associated to each homomorphism $f: M\rightarrow\QOp$
of degree $0$.
In~\cite[Proposition 1.4.7]{FresseCylinder},
we note that this homomorphism $\phi_f$
defines a genuine morphism on the quasi-free operad $\POp = (\FOp(M),\partial_{\theta})$
if and only if we have the relation
\begin{equation}\label{QuasiFreeMorphism}
\delta(f) - \phi_f\cdot\theta = 0
\end{equation}
in $\Hom_{\M}(M,\QOp)$.

\subsubsection{The filtration of quasi-free operads by arity of generators}\label{HomotopySpectralSequence:QuasiFreeFiltration}
In~\S\ref{HomotopySpectralSequence:QuasiFreeOperads},
we explain that the generating object of a free operad $\FOp(M)$
is naturally embedded in $\FOp(M)$.
In fact,
the free operad $\FOp(M)$ has a natural splitting in the category of $\Sigma_*$-objects $\FOp(M) = \bigoplus_{r=0}^{\infty} \FOp_r(M)$
such that $\FOp_0(M) = \IOp$
and $\FOp_1(M) = M$.
Intuitively,
the $\Sigma_*$-object $\FOp_r(M)$
is the submodule of $\FOp(M)$
spanned by $r$-fold composites of generating elements $\xi_i\in M(n_i)$.

In general,
we assume that the homomorphism $\theta: M\rightarrow\FOp(M)$
which determines the twisting derivation of a quasi-free operad $\POp = (\FOp(M),\partial_{\theta})$
satisfies $\partial_{\theta}(M)\subset\bigoplus_{r\geq 2} \FOp_r(M)$.
From now on,
we also assume that the $\Sigma_*$-object $M$
satisfies $M(0) = M(1) = 0$.
In this situation,
we observe in~\cite[\S\S 1.4.9-1.4.10]{FresseCylinder}
that the arity filtration of~$M$
\begin{equation*}
\sk_s M(n) = \begin{cases} M(n), & \text{if $n\leq s$}, \\ 0, & \text{otherwise}, \end{cases}
\end{equation*}
gives a nested sequence of free operads $\FOp(\sk_s M)$
preserved by the twisting derivation of $\POp = (\FOp(M),\partial)$.
Hence,
we have a nested sequence of quasi-free operads
\begin{equation}\label{QuasiFreeFiltration}
\IOp = \sk_1\POp\subset\dots\subset\sk_s\POp\subset\dots\subset\colim_{s}\sk_s\POp = \POp
\end{equation}
such that $\sk_s\POp = (\FOp(\sk_s M),\partial_{\theta})$, where we take the restriction of the twisting derivation of~$\POp$
to $\FOp(\sk_s M)\subset\FOp(M)$.
Furthermore,
we prove in~\cite[Lemma 1.4.11]{FresseCylinder}
that each embedding $i: \sk_{s-1}\POp\hookrightarrow\sk_s\POp$
is an operad cofibration
if $M$ is cofibrant with respect to a standard model structure on~$\Sigma_*$-objects.

The identity $\IOp = \sk_1\POp$ follows from the assumption $M(0) = M(1) = 0$.
Note that the assumption $M(0) = 0$
also implies that the operad $\FOp(M)$
is non-unitary.

\medskip
We study the mapping space $\Map_{\Op_0}(\POp,\QOp)$ associated to a quasi-free operad~$\POp = (\FOp(M),\partial)$
and a fixed operad $\QOp\in\Op_0$ (which is automatically fibrant because every dg-module is so).
We pick a simplicial framing of~$\QOp$
and we take $\Map_{\Op_0}(\POp,\QOp) = \Mor_{\Op_0}(\POp,\QOp^{\Delta^{\bullet}})$
as definition for a mapping space
targeting to $\QOp$.
We have then:

\begin{prop}\label{HomotopySpectralSequence:MappingSpaceDecomposition}
In the setting of~\ref{HomotopySpectralSequence:QuasiFreeFiltration},
the mapping space $\Map_{\Op_0}(\POp,\QOp)$
associated to a quasi-free operad $\POp = (\FOp(M),\partial)$
is the limit term of a tower of fibrations
\begin{equation*}
\cdots\rightarrow\Map_{\Op_0}(\sk_s\POp,\QOp)\rightarrow\Map_{\Op_0}(\sk_{s-1}\POp,\QOp)\rightarrow\cdots
\rightarrow\Map_{\Op_0}(\sk_1\POp,\QOp) = *
\end{equation*}
with the mapping spaces $\Map_{\Op_0}(\FOp(M(s)),\QOp)$
as fibers, for any choice of morphism $\phi: \POp\rightarrow\QOp$ as base point,
where we identify the dg-module $M(s)$
with a $\Sigma_*$-object $M(s)\subset M$
concentrated in arity $s$.
\end{prop}

\begin{proof}
The morphisms $i^*: \Map_{\Op_0}(\sk_s\POp,\QOp)\rightarrow\Map_{\Op_0}(\sk_{s-1}\POp,\QOp)$
induced by the embeddings $i: \sk_{s-1}\POp\hookrightarrow\sk_s\POp$
are fibrations of simplicial sets
because, in any model category, a morphism of mapping spaces
induced by a cofibration on the source
is so.
Moreover, we have clearly $\Map_{\Op_0}(\POp,\QOp) = \Map_{\Op_0}(\colim_s\POp,\QOp) = \lim_s\Map_{\Op_0}(\POp,\QOp)$.

Thus,
we just have to determine the fiber of~$i^*$ over the restriction of a given morphism $\phi: \POp\rightarrow\QOp$.
For this purpose
we use the determination of morphisms on quasi-free operads
in terms of homomorphisms of $\Sigma_*$-objects.
We have $\phi = \phi_f$ for some homomorphism $f: M\rightarrow\QOp$.

Let $\phi_g: \sk_s\POp\rightarrow\QOp^{\Delta^n}$
be an operad morphism towards the term $\QOp^{\Delta^n}$
of the simplicial framing of $\QOp$.
This morphism is determined by a homomorphism of $\Sigma_*$-objects
$g: \sk_s M\rightarrow\QOp^{\Delta^n}$.
For a morphism $\phi_g$ in the fiber of $\phi_f$,
the equation $\phi_g|_{\sk_{s-1}\POp} = \phi_f|_{\sk_{s-1}\POp}$
amounts to the relation $g|_{\sk_{s-1} M} = f|_{\sk_{s-1} M}$
in $\Hom_{\M}(\sk_{s-1} M,\QOp^{\Delta^n})$,
where we apply the constant map $\sigma: \{0<\dots<n\}\rightarrow\{0\}$
to identify $f\in\Hom_{\M}(\sk_{s-1} M,\QOp^{\Delta^0})$
with a homomorphism of $\Hom_{\M}(\sk_{s-1} M,\QOp^{\Delta^n})$.
Let $u\in\Hom_{\M}(M(s),\QOp^{\Delta^n})$
be the homomorphism defined by the difference $g - f$
on $M(s)$.

The homomorphism $g$ is obviously fully determined by the relation $g|_{\sk_{s-1} M} = f|_{\sk_{s-1} M}$
on $\sk_{s-1} M$
and the identity $g = f + u$ on $M(s)$.
Observe now that the equation
\begin{equation}\label{QuasiFreeMorphismReminder}
\delta(g) - \phi_g\cdot\theta = 0
\end{equation}
characterizing morphisms $\phi_g: \sk_s\POp\rightarrow\QOp^{\Delta^n}$
holds in $\Hom_{\M}(\sk_s M,\QOp^{\Delta^n})$
if and only if we have $\delta(u) = 0$
in $\Hom_{\M}(M(s),\QOp^{\Delta^n})$,
and hence if and only if $u$ defines a morphism of dg-modules $u: M\rightarrow\QOp^{\Delta^n}$.
Indeed,
the relation $g|_{\sk_{s-1} M} = f|_{\sk_{s-1} M}$
immediately implies that~(\ref{QuasiFreeMorphismReminder})
holds on $\sk_{s-1} M\subset\sk_{s} M$.
By~\cite[Lemma 1.4.10]{FresseCylinder},
the twisting derivation of~$\POp$
also satisfies $\partial_{\theta}(\sk_s M)\subset\FOp(\sk_{s-1} M)$
when the requirements of~\S\ref{QuasiFreeFiltration}
are fulfilled.
Consequently,
on $M(s)\subset\sk_{s} M$,
equation~(\ref{QuasiFreeMorphismReminder})
reduces to
\begin{equation}
(\delta(g) - \phi_g\cdot\theta)|_{M(s)} = \delta(u) + (\delta(f) - \phi_f\cdot\theta)|_{M(s)} = \delta(u)
\end{equation}
and therefore we have the equivalence $\delta(g) - \phi_g\cdot\theta = 0\Leftrightarrow\delta(u) = 0$.

Thus,
we have a bijective correspondence between operad morphisms $\phi_g: \sk_s\POp\rightarrow\QOp^{\Delta^n}$
such that $\phi_g|_{\sk_{s-1}\POp} = \phi_f|_{\sk_{s-1}\POp}$,
and morphisms of dg-modules $u: M(s)\rightarrow\QOp^{\Delta^n}$,
which are also equivalent to morphisms $\phi_u: \FOp(M(s))\rightarrow\QOp^{\Delta^n}$
on the free operad $\FOp(M(s))$.
Note that this correspondence is obviously natural
with respect to the structure morphisms of the simplicial object $\QOp^{\Delta^{\bullet}}$.
Hence,
as claimed in the proposition,
we have an identity between the simplicial set $\Mor_{\Op_0}(\FOp(M(s)),\QOp^{\Delta^{\bullet}})$
and the fiber over $\phi = \phi_f$
of the morphism $i^*: \Mor_{\Op_0}(\sk_s\POp,\QOp^{\Delta^{\bullet}})\rightarrow\Mor_{\Op_0}(\sk_{s-1}\POp,\QOp^{\Delta^{\bullet}})$.
\end{proof}

\begin{prop}\label{HomotopySpectralSequence:MappingSpaceFibers}
For the free operad $\POp = \FOp(M)$
associated to any cofibrant $\Sigma_*$-object $M$,
we have
\begin{equation*}
\pi_*(\Map_{\Op_0}(\FOp(M),\QOp)) = H_*(\Hom_{\M}(M,\QOp)).
\end{equation*}
\end{prop}

\begin{proof}
The construction of cosimplicial frames of dg-modules,
reviewed in the introduction of this section,
has a straightforward generalization in the category of $\Sigma_*$-objects:
in the definition,
we just replace the tensor product of dg-modules $\otimes: \C\times\C\rightarrow\C$
by the external tensor product of the category of $\Sigma_*$-objects $\otimes: \M\times\C\rightarrow\M$
defined termwise by $(M\otimes D)(n) = M(n)\otimes D$,
for any $M\in\M$ and any $D\in\C$;
the cosimplicial $\Sigma_*$-object $M\otimes N_*(\Delta^{\bullet})$
defined by the tensor product of $M$
with the normalized chain complexes of the simplices $\Delta^{n}$
satisfies clearly $M\otimes N_*(\Delta^{0}) = M$;
if $M$ is a cofibrant $\Sigma_*$-object,
then the morphisms $\eta_i: M\otimes N_*(\Delta^{0})\rightarrow M\otimes N_*(\Delta^{n})$
also satisfy the second requirement of the definition of cosimplicial frames,
because so do the morphisms $\eta_i: N_*(\Delta^{0})\rightarrow N_*(\Delta^{n})$
in the category of dg-modules,
and the external tensor product $\otimes: \M\times\C\rightarrow\M$
preserves colimits on both sides
as well as cofibrations (see~\cite[Lemma 11.4.5]{FresseModules});
the augmentation $\epsilon: M\otimes N_*(\Delta^{n})\rightarrow M\otimes N_*(\Delta^{0})$
is also a weak-equivalence
because the external tensor product $\otimes: \M\times\C\rightarrow\M$
preserves weak-equivalences between cofibrant objects.

By adjunction of model categories,
the cosimplicial free operad $\FOp(M\otimes N_*(\Delta^{\bullet}))$
associated to this cosimplicial frame $M\otimes N_*(\Delta^{\bullet})$
defines a cosimplicial frame of~$\FOp(M)$
in the category of operads.
In the definition of the mapping space $\Map_{\Op_0}(\FOp(M),\QOp)$
we use a simplicial frame of the target $\QOp$,
but we have an isomorphism
\begin{equation*}
\pi_*(\Mor_{\Op_0}(\FOp(M\otimes N_*(\Delta^{\bullet})),\QOp))\simeq\pi_*(\Mor_{\Op_0}(\FOp(M),\QOp^{\Delta^{\bullet}}))
\end{equation*}
by~\cite[Proposition 5.4.7]{Hovey}.
By adjunction,
we also have $\Mor_{\Op_0}(\FOp(M\otimes N_*(\Delta^{\bullet})),\QOp)\simeq\Mor_{\M}(M\otimes N_*(\Delta^{\bullet}),\QOp)$.
The mapping space $\Map_{\M}(M,\QOp) = \Mor_{\M}(M\otimes N_*(\Delta^{\bullet}),\QOp)$
is a simplicial $\kk$-module, just like a mapping space of dg-modules,
with the dg-hom of $\Sigma_*$-objects $\Hom_{\M}(M,\QOp)$
as associated normalized chain complex.
Hence,
we obtain identities
\begin{equation*}
\pi_*(\Map_{\Op_0}(\FOp(M),\QOp))\simeq\pi_*(\Map_{\M}(M,\QOp))\simeq H_*(\Hom_{\M}(M,\QOp))
\end{equation*}
and this achieves the proof of the proposition.
\end{proof}

\subsubsection{Recollections on Bousfield-Kan' extended homotopy spectral sequence}\label{HomotopySpectralSequence:SpectralSequence}
The extended homotopy spectral sequence
of Bousfield-Kan (see~\cite[\S IX.4]{BousfieldKan})
is a spectral sequence in sets
associated to any tower of fibrations
of based simplicial sets
\begin{equation*}
\xymatrix{ X = \lim_{s} X_{s}\ar[r] & \cdots\ar[r] & X_{s}\ar[r] & X_{s-1}\ar[r] & \cdots\ar[r] & X_0\ar[r] & X_{-1} = *\\
&\cdots & F_{s}\ar[u] & F_{s-1}\ar[u] & \cdots & F_{0}\ar[u] & }.
\end{equation*}
Recall simply that the $E_1$-term
of this spectral sequence is defined by the homotopy
of the fibers $F_s$,
so that:
\begin{equation*}
E_1^{s t} = \pi_{t-s}(F_s,*),\quad\text{whenever $t-s\geq 0$, for all $s\geq 0$}.
\end{equation*}

In the next section,
we adapt an analysis of~\cite{BousfieldKan} to determine homotopy groups $\pi_*(\lim_{s} X_{s})$
from $E_1^{s t}$
in a situation where the extended homotopy spectral sequence
degenerates at $E_1$-stage.
For our purpose,
we apply the extended homotopy spectral sequence
to towers of fibrations arising from the mapping space decomposition of Proposition~\ref{HomotopySpectralSequence:MappingSpaceDecomposition}.
In this context,
we have by Proposition~\ref{HomotopySpectralSequence:MappingSpaceDecomposition}
and Proposition~\ref{HomotopySpectralSequence:MappingSpaceFibers}:
\begin{equation*}
E_1^{s t} = H_{t-s}(\Hom_{\M}(M(s),\QOp)),
\end{equation*}
whenever the definition of $E_1^{s t}$ makes sense.
Note that the tower of fibrations of Proposition~\ref{HomotopySpectralSequence:MappingSpaceDecomposition}
really begins at $s=2$
since we have $\sk_0\POp = \sk_1\POp = \IOp$.

\subsubsection*{Remark}
The thesis~\cite{Rezk} gives, in the simplicial setting,
a spectral sequence computing the homotopy of operadic mapping spaces $\Map_{\Op_0}(\POp,\QOp)$
at a base point $\phi\in\Map_{\Op_0}(\POp,\QOp)_0$
from an operadic cohomology $H^*_{\Op_0}(\POp,\QOp)$.
If $\POp$ and $\QOp$ are both discrete operads,
then this spectral sequence can be identified with the extended homotopy spectral sequence associated to a decomposition
of the operadic mapping space $\Map_{\Op_0}(\FOp^{\bullet}(\POp),\QOp)$,
where $\FOp^{\bullet}(\POp)$
refers to the usual cotriple resolution of~$\POp$
in the category of operads.

If $\POp = (\FOp(M),\partial)$ is the quasi-free model of a binary Koszul operad $\HOp$,
equipped with a trivial differential,
and $\QOp$ is also equipped with a trivial differential,
then the term $E_2^{s t}$ of our extended homotopy spectral sequence
can also be identified with an operadic cohomology $H^*_{\Op_0}(\HOp,\QOp)$.
In this setting,
the extended homotopy spectral sequence of~\S\ref{HomotopySpectralSequence:SpectralSequence}
agrees with the spectral sequence of~\cite{Rezk}.

The application and the analysis of such general spectral sequences, computing operadic mapping spaces from operadic cohomology groups,
is quite involved
in the context of $E_n$-operads.
The main results of this paper, proved in the next section,
rely on a basic application of the decomposition of Proposition~\ref{HomotopySpectralSequence:MappingSpaceDecomposition}.

\section{Applications of the extended homotopy spectral sequence}\label{HomotopyGroups}
The goal of this section is to prove Theorem~\ref{Result:MappingSpaces},
the main result of the article.
For this aim,
we apply the homotopy spectral sequence of~\S\ref{HomotopySpectralSequence:SpectralSequence}
to mapping spaces~$\Map_{\Op_0}(\POp,\COp)$
such that $\POp = \BOp^c(\Lambda^{-n}\EOp_n^{\vee})$ is the operadic cobar construction $\BOp^c(-)$
applied to the $n$-fold operadic desuspension $\Lambda^{-n}$
of the dual cooperad $\EOp_n^{\vee}$
of a certain $E_n$-operad~$\EOp_n$.

For the moment,
we assume $n<\infty$.
In fact,
we take the same $E_n$-operad~$\EOp_n$ as in~\cite{FresseEnKoszulDuality},
namely a certain suboperad of the chain Barratt-Eccles operad $\EOp = N_*(E\Sigma_*)$ (see~\cite{BarrattEccles,BergerFresse})
so that $\EOp_n$ is equivalent to the chain operad of little $n$-cubes (see~\cite{BergerFresse}).
For our purpose,
we essentially need to recall that the dg-modules $\EOp_n(r)$
are bounded and finitely generated, for all $r\in\NN$,
and satisfy $\EOp_n(0) = 0$, $\EOp_n(1) = \kk$.
Hence,
the collection $\EOp_n^{\vee}$ of dual objects $\EOp_n^{\vee}(r) = \EOp_n(r)^{\vee}$
in the category of dg-modules inherit a cooperad structure.
In~\cite[Proposition 1.3.5]{FresseEnKoszulDuality},
we also prove that the underlying $\Sigma_*$-object
of this cooperad $\EOp_n^{\vee}$
is cofibrant.

The operadic suspension of a $\Sigma_*$-object $M$
is the $\Sigma_*$-object such that:
\begin{equation*}
\Lambda M(r) = \kk[1-r]\otimes M(r)^{\pm},
\end{equation*}
where $\kk[1-r]$ is a monogeneous dg-module concentrated in degree $1-r$
and the exponent $\pm$
refers to a twist, by the signature of permutations,
of the action of~$\Sigma_r$ on~$M(r)$.
The operadic desuspension is the inverse operation
of the operadic suspension.
The operadic suspensions (and desuspensions) of a cooperad $\DOp$
inherit a cooperad structure.
The cobar construction $\BOp^c(\DOp)$ of a cooperad $\DOp$ satisfying $\DOp(0) = 0$ and $\DOp(1) = \kk$
is a quasi-free operad
\begin{equation*}
\BOp^c(\DOp) = (\FOp(\kk[-1]\otimes\widetilde{\DOp}),\partial),
\end{equation*}
where $\widetilde{\DOp}$ refers to the coaugmentation coideal of~$\DOp$,
the $\Sigma_*$-object such that
\begin{equation*}
\widetilde{\DOp}(r) = \begin{cases} 0, & \text{if $r=0,1$}, \\ \DOp(r), & \text{otherwise}, \end{cases}
\end{equation*}
and the $\Sigma_*$-object $\kk[-1]\otimes\widetilde{\DOp}$
is defined termwise by the tensor products $(\kk[-1]\otimes\widetilde{\DOp})(n) = \kk[-1]\otimes\widetilde{\DOp}(n)$.

For short,
we set $\DOp_n = \Lambda^{-n}\EOp_n^{\vee}$.
For our purpose,
we do not need to review the definition of the twisting derivation
of the cobar construction $\BOp^c(\DOp_n)$.
Note simply that this twisting derivation
satisfies the requirement of~\S\ref{HomotopySpectralSequence:QuasiFreeFiltration}.
In fact,
when we analyze the extended homotopy spectral sequence
of~\S\ref{HomotopySpectralSequence:SpectralSequence},
we obtain immediately:

\begin{lemm}\label{HomotopyGroups:SpectralSequence}
For the operads $\POp = \BOp^c(\DOp_n)$ and $\QOp = \COp$,
the extended homotopy spectral sequence associated to the tower of fibrations of Proposition~\ref{HomotopySpectralSequence:MappingSpaceDecomposition}
satisfies
\begin{equation*}
E_1^{s t} = \begin{cases} \kk, & \text{if $s = 2$ and $t-s = 0$}, \\ 0, & \text{otherwise}, \end{cases}
\end{equation*}
whenever the definition of $E_1^{s t}$
makes sense (in the range $s\geq 2$ and $t-s\geq 0$).
\end{lemm}

\begin{proof}
In~\S\ref{HomotopySpectralSequence:SpectralSequence},
we record that $E_1^{s t} = H_*(\Hom_{\M}(M(s),\QOp)$,
for any pair $\POp = (\FOp(M),\partial)$ and $\QOp\in\Op_0$.
In the case $\POp = \BOp^c(\DOp_n)$ and $\QOp = \COp$,
we obtain immediately $E_1^{s t} = H_{t-s-1}((\DOp_n(s)^{\vee})^{\Sigma_s})$, for all $s\geq 2$,
since $\COp(s) = \kk$
is the trivial representation of $\Sigma_s$.
Hence we have
\begin{equation*}
E_1^{s t} = H_{t-s-1}(\Lambda^n\EOp_n(s)_{\Sigma_s}) = H_{t-s+n(s-1)-1}(\EOp_n(s)_{\Sigma_s})
\end{equation*}
because each dg-module $\EOp_n(s)$ is finitely generated
and the symmetric group $\Sigma_s$
acts freely on $\EOp_n(s)$.

The case $n=1$ is easy, because our $E_1$-operad $\EOp_1$
is identified with the associative operad $\AOp$
and $\AOp(s)$ is the regular representation of the symmetric group $\Sigma_s$, viewed as a dg-module concentrated in degree $0$.
Thus, we focus on cases $n>1$.

In~\cite[Proposition 1.2.8]{FresseEnKoszulDuality},
we observe that computations of~\cite{Cohen} imply that $H_*(\EOp_n(s)_{\Sigma_s})$ vanishes in degree $d>(n-1)(s-1)$,
from which we deduce the identity $E_1^{s t} = 0$ when $s-2>0$ or $t-s>0$.
In the case $s=2$,
the dg-module $\EOp_n(2)$
is identified with a truncation in degree $d\leq n-1$ of the usual free resolution of the trivial $\Sigma_2$-module
of rank $1$
and we have $H_{n-1}(\EOp_n(s)_{\Sigma_s}) = \kk$.
Therefore
we obtain the identity $E_1^{s t} = \kk$ for $s=2$ and $t-s=0$.
\end{proof}

\subsubsection{Analysis of base points}\label{HomotopyGroups:BasePoints}
We study the image of morphisms $\phi: \BOp^c(\DOp_n)\rightarrow\COp$
in $E_1^{2 2} = \pi_0(\Map_{\Op_0}(\sk_2 \BOp^c(\DOp_n),\COp))$.
We focus on cases $n>1$ first.

Recall that the commutative operad $\COp$ is generated by an operation $\mu\in\COp(2)$
which represents the structure product
of commutative algebras.

In the proof of Lemma~\ref{HomotopyGroups:SpectralSequence},
we recall that the dg-module $\EOp_n(2)$
is identified with a truncation in degree $d\leq n-1$ of the usual free resolution of the trivial $\Sigma_2$-module
of rank $1$.
In view of this identity,
the suspended dg-module $\DOp_n(2) = \Lambda^{-n}\EOp_n(2)^{\vee}$
satisfies $H_*(\DOp_n(2)) = \kk$ if $*=1,n$ and $H_*(\DOp_n(2)) = 0$
otherwise.

In~\cite[\S 0.3.2, \S 4.2.1]{FresseEnKoszulDuality},
we define cycles $\mu,\lambda_{n-1}\in\EOp_n(2)$ such that the homology class of $\mu$ (respectively, $\lambda_{n-1}$)
generates $H_*(\EOp_n(2))$ in degree $*=0$ (respectively, $*=n-1$)
and dual basis elements $\mu^{\vee},\lambda_{n-1}^{\vee}\in \DOp_n(2)$
such that the homology class of $\mu^{\vee}$ (respectively, $\lambda_{n-1}^{\vee}$)
generates $H_*(\DOp_n(2))$ in degree $* = n$ (respectively, $* = 1$).
Note that such elements $\mu^{\vee},\lambda_{n-1}^{\vee}\in\DOp_n(2)$
define cocycles in the cobar construction $\BOp^c(\DOp_n)$
because the relation $\partial(\sk_2 M)\subset\FOp(\sk_1 M) = \IOp$,
which holds for the quasi-free operad $\BOp^c(\DOp_n) = (\FOp(\kk[-1]\otimes\widetilde{\DOp}_n),\partial)$,
implies that the twisting derivation of~$\BOp^c(\DOp_n)$
vanishes on $\kk[-1]\otimes\widetilde{\DOp}_n(2)$

By~\cite[Lemma A]{FresseEnKoszulDuality} (see also the review of~\S 4.2 in \emph{loc. cit.}),
we have a morphism $\phi: \BOp^c(\DOp_n)\rightarrow\COp$
such that $\phi_*(\lambda_{n-1}^{\vee}) = \mu$.
The restriction of this morphism to $\sk_2 \BOp^c(\DOp_n) = \FOp(M_{(2)})$, $M = \DOp_n$,
gives a vertex $\phi\in\Map_{\Op_0}(\FOp_{(2)}(M),\COp)$
generating the only non-trivial term $E_1^{2 2} = \kk$
of our homotopy spectral sequence.
(To check this,
apply the identity $E_1^{2 2} = H_{n-1}(\EOp_n(2)_{\Sigma_2})$ used in the proof of Lemma~\ref{HomotopyGroups:SpectralSequence}.)

Note further that each element of~$E_1^{2 2} = \pi_0(\Map_{\Op_0}(\FOp_{(2)}(M),\COp))$
is hit by a morphism $\phi_{c}: \BOp^c(\DOp_n)\rightarrow\COp$, $c\in\kk$,
simply defined by the composite of~$\phi: \BOp^c(\DOp_n)\rightarrow\COp$
with the morphism $\rho_{c}: \COp\rightarrow\COp$
such that $\rho_{c}(\mu) = c\cdot\mu$
for the generating operation of the commutative operad -- the definition of this morphism $\rho_{c}$
involves the convention $\COp(0) = 0$.

These concluding observations also hold in the case $n=1$.
Recall that our $E_1$-operad
is identified with the associative operad $\AOp$
and $\EOp_1(2) = \AOp(2)$
is the regular representation of the symmetric group $\Sigma_2$,
viewed as a dg-module concentrated in degree $0$.
The generating element $\mu\in\AOp(2)$
represents the structure product of associative algebras.
In this case,
we still have a morphism $\phi: \BOp^c(\DOp_1)\rightarrow\COp$
mapping the dual basis elements of $\kk[-1]\otimes\widetilde{\DOp}_1(2) = \AOp^{\vee}(2)$
to the generating operation of the commutative operad $\mu\in\COp(2)$.
Moreover,
the composition of this morphism $\phi: \BOp^c(\DOp_1)\rightarrow\COp$
with the rescaling $\rho_{c}: \COp\rightarrow\COp$
still gives morphisms $\phi_{c}: \BOp^c(\DOp_1)\rightarrow\COp$
so that every element of~$E_1^{2 2}$
is hit by the restriction of a morphism $\phi_{c}$, for some~$c\in\kk$.

\medskip
The result of Lemma~\ref{HomotopyGroups:SpectralSequence}
and this analysis
together give:

\begin{lemm}\label{HomotopyGroups:Abutment}
We have
\begin{equation*}
\pi_0(\Map_{\Op_0}(\BOp^c(\DOp_n),\COp)) = \kk
\quad\text{and}
\quad\pi_i(\Map_{\Op_0}(\BOp^c(\DOp_n),\COp),\phi) = *\quad\text{when $i>0$},
\end{equation*}
for every choice of morphism $\phi: \BOp^c(\DOp_n)\rightarrow\COp$
as base point.
\end{lemm}

\begin{proof}
For short, we set $\POp = \BOp^c(\DOp_n)$.
The lemma is a consequence of the result of Lemma~\ref{HomotopyGroups:SpectralSequence},
the observations of~\S\ref{HomotopyGroups:BasePoints},
and the connectivity lemma of~\cite[\S IX.5]{BousfieldKan}.
In brief:
we take the fibers of the maps
$\Map_{\Op_0}(\sk_s\POp,\QOp)\rightarrow\Map_{\Op_0}(\sk_2\POp,\QOp)$
to obtain a tower of fibrations
satisfying the exact assumptions of~\cite[Chapter IX, Lemma 5.1]{BousfieldKan};
the assertion of this reference implies that the fiber of the map $\Map_{\Op_0}(\POp,\QOp)\rightarrow\Map_{\Op_0}(\sk_2\POp,\QOp)$
is contractible, for any choice of base point in~$\Map_{\Op_0}(\POp,\QOp)$;
the observations of~\S\ref{HomotopyGroups:BasePoints}
imply moreover that any base point of~$\Map_{\Op_0}(\sk_2\POp,\QOp)$
comes from $\Map_{\Op_0}(\POp,\QOp)$;
our claim follows immediately.
\end{proof}

Now,
the main result of~\cite{FresseEnKoszulDuality}
asserts:

\begin{fact}[{see~\cite[Theorem A]{FresseEnKoszulDuality}}]
The operad $\POp_n = \BOp^c(\DOp_n) = \BOp^c(\Lambda^{-n}\EOp_n^{\vee})$
is a cofibrant $E_n$-operad,
for every $n<\infty$.
\end{fact}

Hence,
the result of Lemma~\ref{HomotopyGroups:Abutment}
gives the conclusion of Theorem~\ref{Result:MappingSpaces}
in the case $n<\infty$.

\medskip
In~\cite{FresseEnKoszulDuality},
we also prove that the cooperads $\DOp_n = \Lambda^{-n}\EOp_n^{\vee}$
are connected by morphisms $\sigma^*: \DOp_{n-1}\rightarrow\DOp_n$
such that:

\begin{fact}[{see~\cite[Theorem B]{FresseEnKoszulDuality}}]\label{HomotopyGroups:LimitCobar}
The operad $\POp_{\infty} = \BOp^c(\DOp_{\infty})$
defined by the cobar construction of the colimit cooperad
\begin{equation*}
\DOp_{\infty} = \colim_n\{\DOp_1\xrightarrow{\sigma^*}\cdots\xrightarrow{\sigma_*}\DOp_{n-1}\xrightarrow{\sigma_*}\DOp_{n}\xrightarrow{\sigma^*}\cdots\}
\end{equation*}
is a cofibrant $E_{\infty}$-operad.
\end{fact}

The cobar construction
preserves sequential colimits.
Therefore
we also have $\POp_{\infty} = \colim_n \BOp^c(\DOp_n)$.
Moreover,
the morphisms $\sigma^*: \BOp^c(\DOp_{n-1})\rightarrow \BOp^c(\DOp_n)$
induced by $\sigma^*: \DOp_{n-1}\rightarrow\DOp_n$
are cofibrations of operads (see again~\cite[Proposition 1.3.6]{FresseEnKoszulDuality}).

In arity $r=2$,
the homology morphism $\sigma^*: H_*(\DOp_{n-1}(2))\rightarrow H_*(\DOp_n(2))$
satisfies $\sigma^*(\lambda_{n-2}^{\vee}) = \lambda_{n-1}^{\vee}$.
Hence,
the analysis of~\S\ref{HomotopyGroups:BasePoints}
implies that $\sigma^*$
induces a bijection
\begin{equation*}
\sigma_*: \underbrace{\pi_0(\Map_{\Op_0}(\BOp^c(\DOp_n),\COp))}_{= \kk}
\xrightarrow{\simeq}\underbrace{\pi_0(\Map_{\Op_0}(\BOp^c(\DOp_{n-1}),\COp))}_{= \kk}
\end{equation*}
and we can pass to the limit $n\rightarrow\infty$
in Lemma~\ref{HomotopyGroups:Abutment}
to conclude:

\begin{lemm}\label{HomotopyGroups:AbutmentEInfinityCase}
The result of Lemma~\ref{HomotopyGroups:Abutment}
also holds for $n=\infty$:
we have
\begin{equation*}
\pi_0(\Map_{\Op_0}(\BOp^c(\DOp_{\infty}),\COp)) = \kk
\quad\text{and}
\quad\pi_i(\Map_{\Op_0}(\BOp^c(\DOp_{\infty}),\COp),\phi) = *\quad\text{when $i>0$},
\end{equation*}
for every choice of morphism $\phi: \BOp^c(\DOp_{\infty})\rightarrow\COp$
as base point.\qed
\end{lemm}

This result, together with Fact~\ref{HomotopyGroups:LimitCobar}
gives the conclusion of Theorem~\ref{Result:MappingSpaces}
in the case $n=\infty$
and achieves the proof of this statement.\qed

\section*{Applications}

The homology of an $E_n$-operad~$\POp_n$ is identified, for $n>1$, with the $n$-Gerstenhaber operad~$\GOp_n$,
a composite $\GOp_n = \COp\circ\Lambda^{1-n}\LOp$ of the commutative operad~$\COp$
and of the $(n-1)$-fold operadic desuspension~$\Lambda^{1-n}$
of the Lie operad~$\LOp$ (this identity follows from~\cite{Cohen}, see for instance~\cite[\S 0.3]{FresseEnKoszulDuality}).
The unit morphisms of the operads $\COp$ and $\LOp$
induce obvious embeddings $\COp\hookrightarrow\GOp_n$
and $\Lambda^{1-n}\LOp\hookrightarrow\GOp_n$.
The $n$-Gerstenhaber operad~$\GOp_n$
comes also equipped with an augmentation $\GOp_n = \COp\circ\Lambda^{1-n}\LOp\rightarrow\COp$,
induced by an augmentation on the Lie operad,
and we have an operad embedding
$\Lambda\LOp\hookrightarrow\Lambda^n\COp\circ\Lambda\LOp = \Lambda^n\GOp_n$,
where we use the commutation of operadic suspensions
with composites to obtain the identity $\Lambda^n\COp\circ\Lambda\LOp = \Lambda^n(\COp\circ\Lambda^{1-n}\LOp) = \Lambda^n\GOp_n$.

Recall that the commutative operad~$\COp$
is generated as an operad by the operation~$\mu\in\COp(2)$
representing the structure product of commutative algebras.
The Lie operad~$\LOp$
is generated as an operad by the operation $\lambda\in\LOp(2)$
representing the structure bracket of Lie algebras.
The suspended operad $\Lambda^{1-n}\LOp(2)$
is generated by an operation $\lambda_{n-1}\in\Lambda^{1-n}\LOp(2)$
defined by the $n-1$-fold suspension of~$\lambda\in\LOp(2)$.
The $n$-Gerstenhaber operad $\GOp_n$
can be identified with the operad generated by the generating operations of the commutative operad $\mu\in\COp(2)$
and of the suspension of the Lie operad $\lambda_{n-1}\in\Lambda^{1-n}\LOp(2)$
together with an additional distribution relation between them.

The basis elements~$\mu,\lambda_{n-1}\in\EOp_n(2)$ used in the analysis of~\S\ref{HomotopyGroups:BasePoints}
just define representatives of these generating operations~$\mu\in\COp(2)$ and~$\lambda_{n-1}\in\Lambda^{1-n}\LOp(2)$
in $H_*(\EOp_n) = \GOp_n$.
For the cobar construction $\BOp^c(\DOp_n)$, $\DOp_n = \Lambda^{-n}\EOp_n^{\vee}$,
the existence of a weak-equivalence $\psi: \BOp^c(\DOp_n)\xrightarrow{\sim}\EOp_n$
imply that $H_*(\BOp^c(\DOp_n)) = H_*(\EOp_n) = \GOp_n$.
The cocycle $\mu^{\vee}\in\DOp_n(2)$ (respectively, $\lambda_{n-1}^{\vee}\in\DOp_n(2)$)
considered in the analysis of~\S\ref{HomotopyGroups:BasePoints}
defines a representative of the generating operation $\lambda_{n-1}\in\GOp_n(2)$ (respectively, $\mu\in\GOp_n(2)$)
in $H_*(\BOp^c(\DOp_n))$ (see~\cite[\S 0.3]{FresseEnKoszulDuality}).

Now,
all morphisms $\phi^0,\phi^1: \POp_n\rightarrow\COp$
which are homotopic in the category of operads
induce the same morphism in homology.
To see this,
use the equivalence between left and right-homotopies
for morphisms on a cofibrant operad,
use that path-objects of operads define path-objects of dg-modules -- because limits, weak-equivalences,
and fibrations of operads are created by forget of structure --
and hence that right-homotopies in the category of operads
define right-homotopies in the category of dg-modules.
The morphisms $\phi_{c}: \POp_n\rightarrow\COp$
of~\S\ref{HomotopyGroups:BasePoints},
which define a complete set of representatives of $\pi_0(\Map_{\Op_0}(\POp_n,\COp))$,
satisfy $\phi_{c}(\mu^{\vee}) = 0$
and $\phi_{c}(\lambda_{n-1}^{\vee}) = c\cdot\mu$.
From these observations,
we conclude that the homotopy class of a morphism $\phi: \POp_n\rightarrow\COp$
is fully determined by the associated homology morphism $\phi_*: H_*(\POp_n)\rightarrow\COp$.

The same argument line shows that the same conclusion holds in the case $n=1$,
where we have $\EOp_1 = \AOp$.
By passing to the limit $n\rightarrow\infty$,
we obtain that the homotopy class of a morphism $\phi: \POp_n\rightarrow\COp$
is determined by the associated homology morphism
for $n=\infty$ too.

\medskip
The next theorem gives an application of this analysis
in the particular case of the augmentation morphism $\GOp_n = \COp\circ\Lambda^{1-n}\LOp\rightarrow\COp$:

\begin{mainthm}\label{Result:MorphismToCommutativeHomotopy}
Let $\POp_n$ be any cofibrant $E_n$-operad ($n = 1,2,\dots,\infty$).
All morphisms $\phi^0,\phi^1: \POp_n\rightarrow\COp$
inducing the canonical augmentation $\GOp_n = \COp\circ\Lambda^{1-n}\LOp\rightarrow\COp$
in homology are right-homotopic in the category of operads.\qed
\end{mainthm}

The definition of morphisms $\phi: \BOp^c(\DOp_n)\rightarrow\COp$
satisfying this property in~\cite[\S 1]{FresseEnKoszulDuality}
gives the first step towards the proof that the cobar construction $\BOp^c(\Lambda^n\EOp_n^{\vee})$
is weakly-equivalent to $\EOp_n$
and as such defines a cofibrant replacement of~$\EOp_n$.
The new result of Theorem~\ref{Result:MorphismToCommutativeHomotopy}
implies that each of these morphisms $\phi: \BOp^c(\Lambda^{-n}\EOp_n^{\vee})\rightarrow\COp$
is uniquely determined up to homotopy.

By bar duality of operads (see~\cite{GetzlerJones}),
the existence of a morphism $\phi: \BOp^c(\Lambda^{-n}\EOp_n^{\vee})\rightarrow\COp$, for $n<\infty$,
amounts to the existence of a morphism $\phi^{\sharp}: \Lambda\BOp^c(\Lambda^{-1}\COp^{\vee})\rightarrow\Lambda^n\EOp_n$.
The cobar dual of the desuspension of the commutative cooperad
$\LOp_{\infty} = \BOp^c(\Lambda^{-1}\COp^{\vee})$
is a standard instance of an $L_\infty$-operad,
an operad weakly-equivalent to the operad of Lie algebras $\LOp$.
As such,
this operad satisfies $H_*(\LOp_{\infty}) = \LOp$.

\medskip
Theorem~\ref{Result:MorphismToCommutativeHomotopy}
implies:

\begin{mainthm}\label{Result:MorphismFromLieHomotopy}
Let $\EOp_n$ be an $E_n$-operad which is cofibrant as a $\Sigma_*$-object (but not necessarily cofibrant as an operad)
and so that each dg-module $\EOp_n(r)$ is bounded and finitely generated,
for all $r\in\NN$.
All morphisms $\phi_0,\phi_1: \Lambda\LOp_{\infty}\rightarrow\Lambda^n\EOp_n$
inducing the canonical embedding $\Lambda\LOp\hookrightarrow\Lambda^n\COp\circ\Lambda\LOp = \Lambda^n\GOp_n$
in homology are left-homotopic in the category of operads.
\end{mainthm}

\begin{proof}
The construction of~\cite{BergerMoerdijk}
gives a path-object $\COp^{\Delta^1}$
naturally associated to the commutative operad $\COp$
such that each component of~$\COp^{\Delta^1}$
is a bounded and finitely generated dg-module.
Thus
we can apply the bar duality to this path object
to obtain a cylinder object $\Lambda\LOp_{\infty}\otimes\Delta^1 = \BOp^c((\COp^{\Delta^1})^{\vee})$
associated to $\LOp_{\infty}$.
Under the assumption of the theorem,
the bar duality gives a bijective correspondence between left-homotopies $\phi_{0 1}: \BOp^c((\COp^{\Delta^1})^{\vee})\rightarrow\Lambda^n\EOp_n$
and right-homotopies $\phi^{0 1}: \BOp^c(\Lambda^{-n}\EOp_n^{\vee})\rightarrow\COp^{\Delta^1}$.
Therefore the result of Theorem~\ref{Result:MorphismToCommutativeHomotopy}
implies the assertion of Theorem~\ref{Result:MorphismFromLieHomotopy}.
\end{proof}

Note that the operad $\LOp_{\infty} = \BOp^c(\Lambda^{-1}\COp^{\vee})$
is not cofibrant unless $\QOp\subset\kk$.
The finiteness assumptions can be avoided if we accept to take a cofibrant replacement of $\LOp_{\infty}$
when $\QQ\not\subset\kk$ (this observation follows from standard arguments of homotopical algebra).

\medskip
In the characteristic zero context,
the existence of morphisms $\phi^{\sharp}: \Lambda\LOp_{\infty}\rightarrow\Lambda^n\EOp_n$
is established in~\cite{KontsevichMotives},
for all $n<\infty$.
The motivation of~\cite{KontsevichMotives}
for this construction
is to deduce the definition of deformation complexes
from $E_n$-algebra structures.
The new result of Theorem~\ref{Result:MorphismFromLieHomotopy}
implies that each of the morphisms $\phi^{\sharp}: \Lambda\LOp_{\infty}\rightarrow\Lambda^n\EOp_n$
is uniquely determined up to homotopy.


\begin{thebibliography}{99}

\bibitem{BarrattEccles}
M. Barratt, P. Eccles,
\emph{On $\Gamma_+$-structures. I. A free group functor for stable homotopy theory},
Topology \textbf{13} (1974), 25--45.

\bibitem{BergerFresse}
C. Berger, B. Fresse,
\emph{Combinatorial operad actions on cochains},
Math. Proc. Camb. Philos. Soc. \textbf{137} (2004), 135--174.

\bibitem{BergerMoerdijk}
C. Berger, I. Moerdijk,
\emph{Axiomatic homotopy theory for operads},
Comment. Math. Helv. \textbf{78} (2003), 805--831.

\bibitem{BousfieldKan}
A.Bousfield, D.Kan,
\emph{Homotopy limits, completions and localizations},
Lecture Notes in Mathematics \textbf{304}, Springer Verlag, 1972.

\bibitem{Cohen}
F. Cohen,
\emph{The homology of $\mathcal{C}_{n+1}$-spaces, $n\geq 0$},
\emph{in} ``The homology of iterated loop spaces'',
Lecture Notes in Mathematics \textbf{533}, Springer Verlag (1976), 207--351.

\bibitem{DwyerKanFunction}
W. Dwyer, D. Kan,
\emph{Function complexes in homotopical algebra},
Topology \textbf{19} (1980), 425--440.

\bibitem{DwyerKanLocalization}
\bysame,
\emph{Calculating simplicial localizations},
J. Pure and Appl. Algebra \textbf{18} (1980), 17--35.

\bibitem{FresseModules}
B. Fresse,
\emph{Modules over operads and functors},
Lecture Notes in Mathematics \textbf{1967}, Springer Verlag, 2009.

\bibitem{FresseCylinder}
\bysame, \emph{Operadic cobar constructions, cylinder objects and homotopy morphisms of algebras over operads},
\emph{in} ``Alpine perspectives on algebraic topology (Arolla, 2008)'', Contemp. Math., to appear.

\bibitem{FresseEnKoszulDuality}
\bysame, \emph{Koszul duality of $E_n$-operads},
preprint \href{http://arxiv.org/0904.3123}{\texttt{arXiv:0904.3123}} (2009).

\bibitem{GetzlerJones}
E. Getzler, J. Jones,
\emph{Operads, homotopy algebra and iterated integrals for double loop spaces},
preprint \texttt{arXiv:hep-th/9403055} (1994).

\bibitem{HinichHomotopy}
V. Hinich,
\emph{Homological algebra of homotopy algebras},
Comm. Algebra \textbf{25} (1997), 3291--3323.

\bibitem{Hovey}
M. Hovey,
\emph{Model categories},
Mathematical Surveys and Monographs \textbf{63}, American Mathematical Society, 1999.

\bibitem{KontsevichMotives}
M. Kontsevich,
\emph{Operads and motives in deformation quantization},
Lett. Math. Phys. \textbf{48} (1999), 35--72.

\bibitem{Rezk}
C. Rezk,
\emph{Spaces of algebra structures and cohomology of operads},
PhD Thesis, Massachusetts Institute of Technology, 1996.

\bibitem{Quillen}
D. Quillen,
\emph{Rational homotopy theory},
Ann. Math. \textbf{90} (1969), 205--295.

\bibitem{Tamarkin}
D. Tamarkin,
\emph{Action of the Grothendieck-Teichm\"uller group on the operad of Gerstenhaber algebras},
preprint \href{http://arxiv.org/abs/math/0202039}{\texttt{arXiv:math/0202039}} (2002).

\end{thebibliography}
\end{document}